\documentclass[11pt,reqno]{article} 
\usepackage{amsmath,amssymb,amsthm,mathrsfs}
\newtheorem{theorem}{Theorem}[section]

\newtheorem{corollary}[theorem]{Corollary}

 \theoremstyle{definition}

\newtheorem{remark}[theorem]{Remark}
\usepackage[english]{babel}
\usepackage[utf8]{inputenc}
\usepackage{color}
\usepackage{graphicx}
\usepackage[left=3cm,right=3cm]{geometry}

\def\r{\mathbb R}
\def\s{\mathbb S}
\def\S{\mathsf S}
\def\H{\mathsf H}
\def\P{\mathsf P}
\def\E{\mathsf E}
\def\ee{\mathcal E}
\def\h{\mathbb H}
\def\n{\mathbb N}
\def\z{\mathbb Z}
\def\q{\mathbf q}
\def\e{\mathbf e}

\title{On the stability of Killing cylinders in hyperbolic space}
\author{Antonio Bueno\\
Departamento de Ciencias\\  Centro Universitario de la Defensa de San Javier\\
30729 Santiago de la Ribera, Spain\\
email: antonio.bueno@cud.upct.es\\
\vspace*{.1cm}\\
Rafael L\'opez\\
Departamento de Geometr\'{\i}a y Topolog\'{\i}a\\
 Universidad de Granada. Granada, Spain\\
email: rcamino@ugr.es}
\date{}
 \begin{document}

\maketitle
\begin{abstract}In this paper we study the stability of a Killing cylinder in hyperbolic 3-space when regarded as a capillary surface for the partitioning problem. In contrast with the Euclidean case, we consider a variety of totally umbilical support surfaces, including horospheres, totally geodesic planes, equidistant surfaces and round spheres. In all of them, we explicitly compute the Morse index of the corresponding   eigenvalue problem for the Jacobi operator.  We also address the stability of compact pieces of Killing cylinders with Dirichlet boundary conditions when the boundary is formed by two fixed circles, exhibiting an analogous to the Plateau-Rayleigh instability criterion for Killing cylinders in the Euclidean space. Finally, we prove that the Delaunay surfaces can be obtained by bifurcating Killing cylinders supported on geodesic planes.  
 \end{abstract}

\section{Introduction to the problem}\label{s1}

In this paper we investigate the stability of surfaces that are solutions  for the partitioning problem in the hyperbolic 3-space $\h^3$. The general setting concerning the partitioning problem is analogous to the Euclidean case \cite{ni}, as depicted next: see Fig. \ref{figpartitioning}. Let $W$ be a bounded domain of $\h^3$  with smooth boundary $\partial W$. We are interested in surfaces $\Sigma$ in $\h^3$ with non-empty boundary such that $\mathrm{int}(\Sigma)\subset\mathrm{int}(W)$ and $\partial\Sigma\subset\partial W$. These surfaces are critical points of the area functional among all   surfaces under these conditions that separate $W$ in two domains of prescribed volumes. Any such critical point of this variational problem is called a {\it capillary surface}. A capillary surface $\Sigma$     is characterized by two properties. First, the   mean curvature $H$ of $\Sigma$ is constant. We abbreviate by saying a cmc surface. The second property is  that the contact angle $\gamma$ that makes $\Sigma$ with $\partial W$ along $\partial\Sigma$  is constant.  A surface that is a minimizer up to second order of this problem is said \emph{stable}. If $\Sigma$  is not compact in the partitioning problem, we restrict ourselves to critical points of the area functional with respect to compactly supported volume-preserving variations. 
 
 This paper is motivated, among others, by the pioneering work of   Souam \cite{so} on the study of capillary surfaces in balls of space forms. He proved that a capillary disk  must be a totally geodesic disk or a spherical cap. Also,  if the surface is stable and of genus zero, then it is totally umbilical. However, although the hyperbolic space is a space form, the literature on capillary surfaces in $\h^3$ is limited. Recently, it has been proved that a stable   capillary  surface in a ball \cite{wx}, and more generally in the domain determined by a horosphere \cite{gwx}   is totally umbilical. Other results regarding the isoperimetric problem, free boundary capillary surfaces and symmetry results appear in  \cite{csp,lx,lo1}. 

The partitioning problem in hyperbolic space is richer than in Euclidean space. For example,  besides of totally geodesic planes and geodesic spheres, the umbilical surfaces of $\h^3$ also are equidistant surfaces and horospheres. Moreover, the behavior of cmc surfaces in $\h^3$ is   different than in Euclidean space: cmc surfaces in $\h^3$ with $|H|>1$ find their counterpart in the non-zero cmc surfaces of $\r^3$; cmc surfaces in $\h^3$ with $|H|=1$ are locally isometric to minimal surfaces in $\r^3$; however,   cmc surfaces of $\h^3$ with $0<|H|<1$ have no analogies in Euclidean space.

Besides umbilical surfaces, as far as the authors know, there are no works on explicit examples of capillary surfaces in $\h^3$, even in the case that the support surface is a ball. This contrasts with the great literature on capillary surfaces in a ball of the Euclidean space, initiated with Nitsche \cite{ni}, then with Ros and Vergasta  \cite{rv} and more recently with Fraser and Schoen    \cite{fs}.

The purpose of this paper is to investigate the stability of  Killing cylinders  as capillary surfaces supported on different umbilical surfaces, and eventually to deduce their stability. This is somehow inspired by the large amount of results regarding the stability of circular cylinders in the Euclidean context, since they are simultaneously one of the simplest and more interesting examples of cmc surfaces. Results on stability of cylinders in Euclidean space depending on the  type of support appear in \cite{fr,lo2,lo22,lo3,vo01,vo44,vo5}.  A \emph{Killing cylinder} in $\h^3$  is  defined as the point-set obtained by the movement of a circle by hyperbolic translations of $\h^3$, as a surface of revolution generated by an equidistant line or as  the point-set of equidistant points from a given geodesic, all being equivalent.

Next we detail the organization of the paper and highlight some of the main results. We begin in Sect.  \ref{s2} by setting the notation and basic definitions on the stability of capillary surfaces and define the Morse index. In Sect. \ref{s3} we  introduce  a notation for umbilical surfaces that will be useful throughout this manuscript. As a first result, we prove  that    a horosphere is strongly stable as a capillary surface between parallel planes and between symmetric equidistant surfaces (Thm. \ref{t2}). 

In Sect. \ref{s4} we start to investigate the stability of Killing cylinders. First, we extend to the hyperbolic setting the analogous to the Plateau-Rayleigh instability criterion of cylinders in the Euclidean space \cite{pl,ra}. For that matter, we consider bounded pieces of Killing cylinders of radius $R$ bounded by two circles at distance $T$. Our first result shows that if $T>2\pi\sinh R$, then the Killing cylinder is not stable (Thm. \ref{t3}). We also show this result when the Killing cylinder is regarded as a capillary surface between two parallel totally geodesic planes, and between two parallel horospheres (Thm. \ref{t4}).

In Sect. \ref{sb} we consider a Killing cylinder inside a ball of $\h^3$ and exhibit in Thm. \ref{t-ib} the explicit computation of the Morse index. We show that the Morse index of a Killing cylinder approaches to $1$ as the Killing cylinder tends to the critical position of being tangent to the ball, and then grows to $\infty$ when the radius of the cylinder goes to $0$. 

Following this approach,  in Sect. \ref{s5} we extend this study in case that the Killing cylinder is included in the unbounded domain determined by a horosphere, a totally geodesic plane and an equidistant surface, and its boundary is supported in these supports: Thms. \ref{t6}, \ref{t7} and \ref{t8}.  In all these three cases, the Killing cylinder is not stable. In fact we will derive that the index of these truncated Killing cylinders goes to $\infty$ as the lengths of the   cylinders go to infinity. Finally, in Sect.  \ref{s6} we use the Crandall and Rabinowitz's bifurcation by simple eigenvalues to show that  the Delaunay surfaces of $\h^3$  can be obtained via bifurcating Killing cylinders (Thm. \ref{t-b}).

\section{Stability of capillary surfaces}\label{s2}

In this section we introduce the concept of the Morse index of a capillary surface. In what follows, we adopt \cite{rs,so} in our context.  Fix a domain $W\subset\h^3$ with smooth boundary $\partial W$, let $\Sigma$ be an orientable smooth surface with non empty boundary $\partial\Sigma$ and $\Psi\colon\Sigma\to W$ an immersion which is smooth on the interior of $\Sigma$ and of class $C^2$ up to $\partial\Sigma$. The immersion $\Psi$ is said to be  \emph{admissible} if   $\Psi(\text{int}(\Sigma))\subset\text{int}(W)$ and $\Psi(\partial\Sigma)\subset \partial W$. Suppose that $\Psi(\mbox{int}(\Sigma))$ separates $W$ in two connected components, each having as boundary the union of $\Psi(\mbox{int}(\Sigma))$ and a domain on $\partial W$. Fix one of these components, say  $D$, which it is assumed to be bounded and let $\Omega=\partial D\cap \partial W$. Let $N$ be the unit normal vector field on $\Sigma$ along $\Psi$ pointing into $D$, and let $\overline{N}$ be the unit normal of $\Omega$ pointing outwards $D$. Finally, denote by $\nu$ the exterior unit conormal vector to $\partial\Sigma$ in $\Sigma$ and by $\overline{\nu}$ the exterior unit conormal vector to $\partial\Sigma$ in $\Omega$.  See Fig. \ref{figpartitioning}.

\begin{figure}[h]
\begin{center}
\includegraphics[width=.4\textwidth]{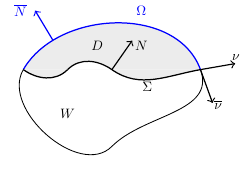}
\caption{The partitioning problem and the vector fields involved in its definition.}
\label{figpartitioning}
\end{center}
\end{figure}

An {\it admissible variation} of $\Psi$ is a smooth map $\Phi:\Sigma\times (-\epsilon,\epsilon)\to \h^3$ such that $\Phi(p,0)=\Psi(p)$ and the maps $\Phi_t\colon p\mapsto\Phi(p,t)$ are   admissible immersions of $\Sigma$ in $W$ for all $|t|<\epsilon$. For a fixed $t$, define $\Sigma_t=\Phi_t( \Sigma)$ and $\Omega(t)$ the domain bounded by $\partial\Sigma_t$ in $\partial W$. We will assume that all variations have compact support, that is, for each $t$ there is a compact set $K\subset \Sigma$ such that $\Phi_t(p)=\Psi(p)$ for all $p\in\Sigma\setminus K$. 
 
Given $\gamma\in (0,\pi)$, define the energy functional $E\colon (-\epsilon,\epsilon)\to\r$ given by 
$$E(t)=|\Sigma_t|-\cos\gamma\, |\Omega(t)|,$$
where $|\Sigma_t|$ is the area of $\Sigma$ induced by $\Phi_t$ and $|\Omega(t)|$ is the area of $\Omega(t)$. Let $V(t)$ be the volume of the domain determined by $\Phi_t$ in $\h^3$,
$$
V(t)=\int_{[0.t]}|\mbox{Jac}\,\Phi|dV,
$$
being $dV$ the volume element in $W$. The variation is said to preserve the volume if $V(t)=V(0)$ for every $t$. The formulas of the first variation of $E(t)$ and $V(t)$ are, respectively, 
$$E'(0)=-2\int_\Sigma H u +\int_{\partial \Sigma} \langle \nu-\cos\gamma\, \bar\nu,\xi\rangle \,,\quad 
V'(0)=\int_{\Sigma} u  ,$$
where $H$ is the mean curvature of the immersion $\Psi$,  $\xi$ is the variation vector field of $\Phi$ and $u= \langle   N,\xi\rangle$. As a consequence, $\Psi$ is a critical point of $E$ for all compactly supported volume-preserving admissible variations  if and only if $H$ is constant and the angle between    $\Sigma$ and $\partial W$  is $\gamma$ along $\partial\Sigma$. Here $\cos\gamma=\langle N,\overline{N}\rangle=\langle\nu,\overline{\nu}\rangle$. We say then that $\Psi$ is a    {\it capillary surface}   in $W$.  Standard arguments prove that for any smooth function $u$ on $\Sigma$ such that $\int_\Sigma u=0$, there is an admissible variation of $\Psi$ preserving the volume with $u=\langle N,\xi\rangle$ \cite{as,bce}. In the above setting, we can also assume that $\partial W$ is formed by different components $\partial W_1,\ldots,\partial W_k$. In such a case,   the components of   $\partial\Sigma$ are included in  some $\partial W_i$, $1\leq i\leq k$. Then, the energy $E$ is modified adding different terms of type $\cos\gamma_i |\Omega_i(t)|$, being $\Omega_i(t)=\partial D\cap\partial W_i$.  Consequently, a critical point satisfies that the angle $\gamma_i$ that makes $\Sigma$ along each one its boundary components with $\partial W_i$ is constant,  $1\leq i\leq k$.

For all compactly supported volume-preserving admissible variations  of $\Psi$, the second variation of $E$ is 
 \begin{equation}\label{e0}
  E''(0)=  -\int_{\Sigma} u\left(\Delta u +(|A|^2-2) u\right)\, + \int_{\partial\Sigma} 
 u\left( \frac{\partial u}{\partial \nu} - {\bf q}\, u\right)\,  ,
 \end{equation}
where $\Delta$ is the Laplacian for the metric induced by $\Psi$ on $\Sigma$, $A$ is the second fundamental form of $\Psi$, $\frac{\partial u}{\partial\nu}=\langle\nabla u,\nu\rangle$ is the conormal derivative of $u$ in the direction of $\nu$ and  
\begin{equation}\label{qq}
 {\bf q}= \frac{1}{\sin\gamma} \overline{A} (\bar\nu,\bar\nu)+ \cot\gamma A(\nu,\nu). 
\end{equation}
Here, $\overline{A}$ is the second fundamental form of $\partial W$ with respect to $-\overline{N}$.   A capillary surface is \emph{stable} if $E''(0)\geq0$ for every admissible variation preserving the volume. If we drop the volume-preserving condition, then the capillary surface is said to be \emph{strongly stable}.

 Integrating by parts, the formula \eqref{e0} defines the quadratic form 
\begin{equation}\label{q}
 Q[u] =E^{''}(0)= \int_\Sigma  \left(|\nabla u|^2 -(|A|^2-2) u^2\right)- \int_{\partial\Sigma} 
 \q u^2 \, .
 \end{equation}
 Thus stability is equivalent to $Q[u]\geq 0$ for all functions of the space $\mathcal{M}=\{u\in C^\infty(\Sigma): \int_\Sigma u=0\}$.  A direct consequence of \eqref{q} is the following result. 

\begin{corollary} \label{c1}
Let $\Sigma$ be a capillary surface on $\partial W$. If $|A|^2\leq 2$ and $\q\leq 0$ along $\partial\Sigma$, then $\Sigma$ is stable. 
\end{corollary}

If $\Sigma$ is not stable, it is interesting to study the subspace of $\mathcal{M}$ where $Q$ is negative definite. The  {\it weak Morse index} of $\Sigma$, denoted by $\mbox{index}_w(\Sigma)$, is defined as the maximum dimension of any subspace of $\mathcal{M}$ on which $Q$ is negative definite.   The weak Morse index tells us the number of directions whose variations decrease the energy of the surface.  The parenthesis in the first integral of \eqref{e0} defines the Jacobi operator $J= \Delta + |A|^2-2 $. It is known that $J$ is a second-order linear elliptic differential operator. If $\Sigma$ is compact, the weak Morse index of $\Sigma$ coincides with the number of negative eigenvalues $\lambda_w$  of the   eigenvalue problem     
\begin{equation}\label{eq10}
\left\{\begin{split}
Ju+\lambda_w u=0&\mbox{ in}\ \Sigma,\\
\frac{\partial u}{\partial \nu} - \q u=0& \mbox{ in}\  \partial\Sigma,\\
u&\in\mathcal{M}
\end{split}\right.
\end{equation}
The fact that the function $u$ belongs to $\mathcal{M}$ is difficult to manage in general. For this reason, instead of \eqref{eq10}, it is better to consider the eigenvalue problem
\begin{equation}\label{eq1}
\left\{\begin{split}
Ju+\lambda u=0&\mbox{ in}\ \Sigma,\\
\frac{\partial u}{\partial \nu} - \q u=0& \mbox{ in}\  \partial\Sigma,\\
u&\in C^\infty_0(\Sigma).\\
\end{split}\right.
\end{equation}
By the ellipticity of $J$, it is well-know that the eigenvalues of \eqref{eq1} (also $\lambda_w$ of \eqref{eq10}) are ordered as a discrete spectrum  $\lambda_1<\lambda_2 \leq\lambda_3\cdots\nearrow \infty$ counting multiplicity  \cite{be,bgm}. The eigenvalues have associated  an orthonormal basis (in the $L^2$-sense) of eigenfunctions $\{u_n\}_{n\in\n}$.  The {\it Morse index} of $\Sigma$, denoted $\mbox{index}(\Sigma)$, is the number of negative eigenvalues of \eqref{eq1} and the {\it nullity} of $\Sigma$ is the dimension of the subspace associated to the eigenvalue $0$, in case that $0$ is an eigenvalue. If $\mbox{index}(\Sigma)=0$, we say that the surface is {\it strongly stable} and this is equivalent to $Q[u]\geq 0$ regardless the condition that $\int_\Sigma u=0$. Both indices are related by the inequalities   
\begin{equation}\label{ii}
\mbox{index}_w(\Sigma)\leq \mbox{index}(\Sigma)\leq \mbox{index}_w(\Sigma)+1.
\end{equation}
 Some interesting cases are the following:
\begin{enumerate}
\item If $\lambda_1\geq 0$, then    the surface is stable because  $\mbox{index}(\Sigma)=0$.  In other words, strong stability implies stability. 
\item If $\lambda_2<0$. Then the subspace spanned by  two eigenfunctions of $\lambda_1$ and $\lambda_2$ has at least dimension $2$. This implies that there is a function $u\in\mathcal{M}$ with $Q[u]<0$. In particular,   $\Sigma$ is not stable.  One can also argue saying that $2\leq \mbox{index}_w(\Sigma) +1$, so $\mbox{index}_w(\Sigma)\geq 1$.
\end{enumerate}
  For   a relation between the weak Morse index and the Morse index, we refer the reader to \cite{koi,vo12,vo2};  see also  the recent papers \cite{tz1,tz2}.

\begin{remark}\label{res}
 If one considers the stability problem of cmc surfaces with fixed boundary, the quadratic form $Q$  is as in \eqref{q} without the boundary term. In the corresponding eigenvalue problems \eqref{eq10} and \eqref{eq1}, the Robin boundary condition is replaced by $u=0$ on $\partial\Sigma$.
\end{remark}
 
 In case that $\Sigma$ is not compact, then we need to take an exhaustion of the surface, similarly as in the  theory of complete minimal surfaces \cite{fs1}.  If $\Sigma_1\subset\Sigma_2\subset\ldots\subset  \Sigma$ is an exhaustion of $\Sigma$ by bounded subdomains, the weak Morse index and the Morse index of $\Sigma$ are defined by 
 \begin{equation}\label{stable-infinito}
 \mbox{index\,}_w(\Sigma)=\lim_{n\to\infty}\mbox{index\,}_w(\Sigma_n),\quad \mbox{index\,}(\Sigma)=\lim_{n\to\infty}\mbox{index\,}(\Sigma_n).
 \end{equation}
We point out that   for two bounded open subdomains $\Sigma'$, $\Sigma''$ of $\Sigma$, if $\Sigma'\subset\Sigma''$, then $\mbox{index\,}_w(\Sigma')\geq \mbox{index\,}_w(\Sigma'')$ and $\mbox{index\,}(\Sigma')\geq \mbox{index\,}(\Sigma'')$.  Notice that  $\mbox{index\,}_w(\Sigma)$ and $\mbox{index\,}(\Sigma)$  are independent of the choice of the exhaustion of $\Sigma$ \cite{fs}. Both numbers can be infinite, but if they are finite, then the relation \eqref{ii} holds too.

\section{First examples of capillary surfaces}\label{s3}
In this paper we consider the upper half space model of $\h^3$, that is,   $\r^{3}_+=\{(x,y,z)\in\r^{3}\colon z>0\}$, endowed with the metric
$$\langle\cdot,\cdot\rangle=\frac{dx^2+dy^2+dz^2}{z^2}.$$
  We   use the affine notions of vertical and horizontal to say parallel to the $z$-axis or parallel to the $xy$-plane, respectively. Denote by $\{\e_1,\e_2,\e_3\}$ to the canonical basis of $\r^3$. 

Since   the metric $\langle\cdot,\cdot\rangle$ is conformal to the Euclidean metric $\langle\cdot,\cdot\rangle_e$  of $\r^3_{+}$, the Levi-Civita connections  $\nabla$ and $\nabla^e$ of $\h^3$ and $\r^3$, respectively, are related by
\begin{equation}\label{r}
\nabla_XY=\nabla^e_XY-\frac{X_3}{z}Y-\frac{Y_3}{z}X+\frac{\langle X,Y\rangle_e}{z}\e_3,
\end{equation}
for two vector fields $X$ and $Y$. If $\Sigma$ is   an immersed surface in $\r^3_+$ and   $N_e$ is its unit normal  computed with the Euclidean metric, then $N=zN_e$ is a unit normal vector of $\Sigma$ in $\h^3$. Furthermore, if $\kappa_i$ are the Euclidean principal curvatures and $\overline{\kappa_i}$ the hyperbolic ones, then from \eqref{r} we deduce $\overline{\kappa_i}=z\kappa_i+(N_e)_3$, where $(\cdot)_3$ stands for the third coordinate. As a consequence, the Euclidean mean curvature $H_e$ and the hyperbolic one $H$ of $\Sigma$ are related by
$$H=zH_e+(N_e)_3=zH_e+\frac{N_3}{z}.$$

Next we show the umbilical surfaces of $\h^3$ since they will act as support surfaces of the capillary surfaces. The relation between the Euclidean and hyperbolic principal curvatures imply that  the umbilical surfaces of $\h^3$   are, as point-sets,  the intersection of the umbilical surfaces of Euclidean space $\r^3$ with $\r^3_{+}$. As a consequence, the umbilical surfaces of $\h^3$ are the following after an appropriate isometry of $\h^3$. 
\begin{enumerate}
\item Totally geodesic planes. They are vertical planes and hemispheres. For $\tau>0$, let $\P_{\tau}=\{(x,y,z)\in\r^3_{+}: y=\tau\}$ and $\S_\tau=\{(x,y,z)\in \r^3_{+}: x^2+y^2+z^2=\tau^2\}$.
\item Horospheres. They are horizontal planes.   For $\tau>0$, let  $\H_\tau=\{(x,y,z)\in\r^3_{+}: z=\tau\}$.  
\item Equidistant surfaces. They are slopped planes.  For  $\theta\in (-\pi/2,\pi/2)$, let $\E_\theta=\{(x,y,z)\in\r^3_{+}: z=y\tan\theta\}$.  
\item Geodesic spheres. They are spheres included in $\r^3_{+}$. For $0<\rho<c$, let $\S(c,\rho)=\{(x,y,z)\in\r^3_+:x^2+y^2+(z-c)^2=\rho^2\}$.
\end{enumerate}
Given two totally geodesic planes $\P_{\tau_1}$ and $\P_{\tau_2}$, or two horospheres $\H_{\tau_1}$ and $\H_{\tau_2}$,  we say that they are parallel and the domain between them is called a \emph{slab}. Similarly, the domain between the equidistant surfaces $\E_{\theta}$ and $\E_{-\theta}$ is called a \emph{wedge}.

We study the stability of pieces of horospheres as capillary surfaces supported on totally geodesic planes and equidistant surfaces. In the following result we show a particular case of Cor. \ref{c1} when $\Sigma$ is a horosphere, since $|A|^2=2$ and $\q=0$ along $\partial\Sigma$.  However, we give a direct proof  by   computing the eigenvalues of \eqref{eq1} and showing that all are positive. Not only for this reason. In the course of the proof of the following theorem, we will use a series of arguments to solve a differential equation by separation of variables. The same type of techniques will be used in the results of the rest of the article with some modifications in the equation  and in the boundary conditions. Therefore, the proof of the following theorem illustrates how to proceed in the successive results.

\begin{theorem}\label{t2}
A horosphere $\H_\tau$ is strongly stable, viewed as a capillary surface in a slab between two   totally geodesic planes, or in a wedge of two   equidistant surfaces.
\end{theorem}

\begin{proof} We denote by $W$ be the slab or    the wedge and by $\Sigma$   the portion of $\H_\tau$ in $W$. See Fig. \ref{fig2}.  After an isometry of $\h^3$ we suppose that the slab is determined by $\P_{-1}\cup\P_1$ and the wedge by  $\E_{-\theta}\cup\E_\theta$. So, the slab and the wedge are   $W=\{(x,y,z)\in\r^3_+:|y|<1\}$ and $W=\{(x,y,z)\in\r^3_+:|y|<z/\tan\theta\}$, respectively. In both cases, we will choose as the domain $D$ the part of $W$ above $\Sigma$, $D=W\cap \{(x,y,z)\in\r^3_+:z>\tau\}$. In particular, $N=\tau\e_3$ and $A=\mbox{id}$ on $\Sigma$.

\begin{figure}[h]
\begin{center}
 \includegraphics[width=.4\textwidth]{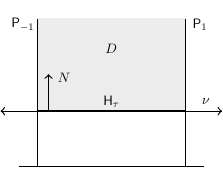}\quad \includegraphics[width=.4\textwidth]{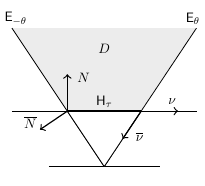}
\end{center}
\caption{Left: a horosphere supported in two vertical planes. Right: a horosphere supported in two equidistant surfaces.}\label{fig2}
\end{figure}

First, assume that $W$ is a slab. Then $\gamma=\pi/2$, so $\q=0$ because $\overline{A}=0$. For the computation of the eigenvalues of \eqref{eq1}, we parametrize $\Sigma$ by  $\Psi(t,s)=(t,s,\tau)$, $(t,s)\in \r\times (-1,1)$. Since $\Sigma$ is not compact,  and such as we exhibited in the last paragraph of Sect. \ref{s2}, it is enough to consider an exhaustion by compact subdomains of $\Sigma$. The compact subdomains are  $\Sigma_T=\Psi((0,T)\times (-1,1))=\Sigma\cap\{(x,y,z)\in\r^3_{+}: 0<x<T\}$, for $T>0$. The eigenvalue problem \eqref{eq1} must be now modified by 
 \begin{equation}\label{eqd0}
\left\lbrace\begin{array}{ll}
Ju+\lambda u=0 & \mathrm{in }\ \Sigma_T,\\
\frac{\partial u}{\partial \nu} =0 & \mathrm{on }\ \partial \Sigma_T\cap \P_\epsilon,\\
u=0& \mathrm{at }\ t=0,T.
\end{array}\right.
\end{equation}
We use the method of separation of variables by taking $u(t,s)=f(t)g(s)$. The coefficients of the first fundamental form with respect to $\Psi$ are $g_{11}=g_{22}=1/\tau^2$ and $g_{12}=0$. Then $\Delta=\tau^2(\partial_t^2+\partial_s^2)$. Since $\nu_\epsilon=\epsilon\tau\Psi_s$, we deduce 
$$\langle\nabla u,\nu_\epsilon\rangle=\tau^2\langle f'g\Psi_t+fg'\Psi_s,\nu_\epsilon\rangle=\tau^2fg'\langle\Psi_s,\nu_\epsilon\rangle=\epsilon\tau fg'.$$
Therefore, \eqref{eqd0} is equivalent to
 \begin{equation}\label{eqd00}
\left\lbrace\begin{array}{ll}
 f''g+fg''  + \frac{1}{\tau^2}\lambda fg=0 & \mathrm{in}\ (0,T)\times (-1,1),\\
g'(\pm 1)=0, &  \\
  f(0)= f(T)=0.&
\end{array}\right.
\end{equation} 
From the first equation, we get 
$$\frac{f''}{f}=\mu=-\frac{g''}{g}-  \frac{1}{\tau^2}\lambda,$$
for a certain constant $\mu$. The last boundary condition in \eqref{eqd00} implies $\mu<0$, and the solution for $f$ is
$$f(t)=c_1 \sin\left(\frac{m\pi}{T}t\right), \quad \mu=-\left(\frac{m\pi}{T}\right)^2,\  m\in\n.$$
Hence $g$ satisfies
$$g''+\left(\frac{\lambda}{\tau^2}-\left(\frac{m\pi}{T}\right)^2\right)g=0$$
with boundary conditions $g'(\pm 1)=0$. Let us distinguish cases.
\begin{enumerate}
\item Case $\frac{\lambda}{\tau^2}-(\frac{m\pi}{T})^2=0$. Then $g(s)=c_2$ is a constant function and consequently, 
$$u_m(t,s)=c_3 \sin\left(\frac{m\pi}{T}t\right),\quad  \lambda_m=\left(\frac{m\pi\tau}{T}\right)^2,\ m\in\n.$$
\item Case $\frac{\lambda}{\tau^2}-(\frac{m\pi}{T})^2=-\delta^2<0$. Now 
$$g(s)=c_2 \sinh(\delta s)+c_3  \cosh(\delta s).$$
The boundary conditions imply $c_3=0$ and $\delta\cosh\delta=0$. Since this equation has no positive solutions $\delta$, this case is not possible.
\item Case $\frac{\lambda}{\tau^2}-(\frac{m\pi}{T})^2=\delta^2>0$. The solution is 
$$g(s)=c_2\sin(\delta s)+c_3\cos(\delta s).$$
The boundary conditions $g'(\pm 1)=0$ yield $c_2=0$ and $\sin\delta=0$. Thus $\delta=n\pi$, $n\in\n$. Then the function $g$ is $g_n(s)=c_3\cos(n\pi s)$. Consequently, the eigenfunctions and eigenvalues are
$$u_{m,n}(t,s)= \sin\left(\frac{m\pi}{T}t\right) \cos(n\pi s),$$
$$ \lambda_{m,n}=\tau^2\left(n^2\pi^2 +\left(\frac{m\pi}{T}\right)^2\right),\quad n,m\in\n.$$
\end{enumerate}
This shows the positiveness of the eigenvalues of $\Sigma_T$, hence each $\Sigma_T$ is strongly stable. Taking limits as $T\rightarrow\infty$ we also conclude that $\Sigma$ is strongly stable thanks to \eqref{stable-infinito}.

Assume now that $W$ is a wedge.    Then we consider again the subdomains $\Sigma_T$, $T>0$, of $\Sigma$. We  have $\nu_\epsilon= \epsilon\tau \e_2$, 
 $$\overline{N}_1=z(\sin\theta \e_2-\cos\theta \e_3),\quad \overline{N}_{-1}=-z(\sin\theta \e_2+\cos\theta \e_3),$$
and  $A(\nu_\epsilon,\nu_\epsilon)=1$. On the other hand, 
 $$\bar{\nu}_1=-\tau(\cos\theta \e_2+\sin\theta \e_3),\quad \overline{\nu}_{-1}=\tau(\cos\theta \e_2-\sin\theta \e_3)$$
 and  $\overline{A}(\overline{\nu}_\epsilon,\overline{\nu}_\epsilon)= \cos\theta$. Then 
\[ \q= \frac{1}{\sin\theta} \overline{A} (\overline\nu,\overline\nu)- \cot\theta A(\nu,\nu)=\frac{\cos\theta}{\sin\theta} - \cot\theta=0 . 
\]
The eigenvalue problem coincides with   \eqref{eqd00}, obtaining the same conclusion. 
\end{proof}

\section{The Plateau-Rayleigh instability criterion}\label{s4}

In this section we will extend to the hyperbolic setting the Plateau-Rayleigh instability criterion of cylinders in Euclidean space \cite{pl,ra}, which asserts that a circular cylinder of radius $r$ and length $T$ is unstable whenever $T>2\pi r$. For that matter, we investigate compact pieces of Killing cylinders in $\h^3$ bounded by two circles, and explicitly compute how the separation between these circles establishes the instability of the Killing cylinder. First, we consider the instability in the case that the boundary consists in two fixed circles and next, when the boundary is contained in a pair of totally geodesic planes or a pair of horospheres.
 
We introduce some notation that will be useful hereinafter. Let $\Gamma_R$ be a circle of radius $R>0$, which   can be assumed to be parametrized by $\theta\mapsto (r\cos \theta, r\sin \theta,1)$, where $r=\sinh R$. If $L$ is the $z$-axis, then $\Gamma_R$ is invariant by the rotations of $\h^3$ about $L$.  Let $O=(0,0,0)$ be the intersection point of $L$ with $\{z=0\}$. The image of $\Gamma_R$ under the hyperbolic translations from $O$ defines the  Killing cylinder  $C_R$ of radius $R$. These hyperbolic translations are      Euclidean homotheties from $O$.  Thus a parametrization   of $C_R$ is 
\begin{equation}\label{p}
\Psi(t,\theta)=e^t\left(r\cos {\theta} ,r\sin {\theta} ,1\right),\quad t\in\r,\ \theta\in[0,2\pi].
\end{equation}
The unit normal of $C_R$ is
\begin{equation}\label{n}
N=\frac{e^{t}}{\sqrt{1+r^2}}\left(-\cos\theta,-\sin\theta,r\right),
\end{equation}
and the principal curvatures are  
$$\kappa_1= \frac{\sqrt{1+r^2}}{r},\quad \kappa_2=  \frac{r}{\sqrt{1+r^2}}.$$
Therefore
\begin{equation*}
\begin{split}
H&=\frac{1+2r^2}{2r\sqrt{1+r^2}}=\frac{\tanh R+\coth R}{2},\\
|A|^2&=\frac{(1+r^2)^2+r^4}{r^2(1+r^2)}=1+\tanh ^2R+\text{csch}^2R.
\end{split}
\end{equation*}
On the other hand,   the coefficients of the first fundamental form are $g_{11}=1+r^2$, $g_{12}=0$ and $g_{22}=r^2$. Then the Laplace-Beltrami operator is 
$$\Delta=\frac{1}{1+r^2}\partial_t^2 +\frac{1}{r^2}\partial_{\theta}^2,$$
and the Jacobi operator has the expression  
$$J=\Delta+|A|^2-2=  \frac{1}{1+r^2}\partial_t^2 +\frac{1}{r^2}\partial_{\theta}^2+\frac{1}{r^2(1+r^2)}.$$
Let 
$$\varpi=\frac{1}{r^2(1+r^2)}=\frac{1}{\sinh^2R\cosh^2R}.$$
Finally, we define $C_R^T$ as a piece of a Killing cylinder of radius $R>0$ and length $T$ parametrized as in \eqref{p} but now the variable $t$ varies in the interval $[0,T]$.

The first result concerns the stability of pieces of Killing cylinders with fixed boundary in the sense of  Rem. \ref{res}.

\begin{theorem}\label{t3} Let   $R,T>0$ and define
\begin{equation}\label{eqindex}
\eta(T)=\max\{m\in\n\colon m<\frac{T}{\pi\sinh R}\}.
\end{equation}
Consider the eigenvalue problem \eqref{eq1} of $C_R^T$ with Dirichlet boundary condition. Then, $\mathrm{index}(C_{R}^T)=\eta(T)$. In consequence,
\begin{enumerate}
\item if $T>\pi\sinh R$, then $C_R^T$ is not strongly stable;
\item if $T>2\pi\sinh R$, then $C_R^T$ is not stable.
\end{enumerate}
As a consequence the Killing cylinder $C_R$ is not stable.
\end{theorem}

\begin{proof}      We know that  $C_R^T=\Psi([0,T]\times[0,2\pi])$. Consider   the eigenvalue problem with Dirichlet boundary conditions
\begin{equation}\label{eqd}
\left\lbrace\begin{array}{ll}
Ju+\lambda u=0 & \mathrm{in}\ C_R^T,\\
u=0 & \mathrm{on}\ \partial C_R^T.
\end{array}\right.
\end{equation}
We use the method of separation of variables following similar steps as in Thm. \ref{t2}. Let   $u(t,\theta)=f(t)g(\theta)$, $t\in [0,T]$ and $\theta\in [0,2\pi ]$. Then the differential equation of \eqref{eqd} is 
\[\frac{1}{1+r^2}f''g+\frac{1}{r^2}fg''+(\varpi+\lambda)fg=0,\]
or equivalently,
$$
 \frac{f''}{(1+r^2)f}+\frac{g''}{r^2g}+\varpi+\lambda=0.
$$
Thus there is a constant $\mu\in\r$ such that
$$\frac{g''}{r^2 g}=\mu=-\frac{f''}{(1+r^2)f}-(\varpi+\lambda).$$
In particular, $g$ is an eigenfunction of the Laplacian on $\s^1$. This implies $\mu=-n^2/r^2$, $n\in\n$, and
\begin{equation}\label{gn}
g_n(\theta)=c_1 \cos( n\theta )+c_2 \sin( n\theta ),\quad c_1,c_2\in\r.
\end{equation} For the function $f$, we have 
\begin{equation}\label{fn}
f''+(1+r^2)(\varpi-\frac{n^2}{r^2}+\lambda)f=0.
\end{equation}
By the boundary conditions $f(0)=f(T)=0$, it is necessary that $(1+r^2)(\varpi-\frac{n^2}{r^2}+\lambda)=\delta^2>0$. Thus the solution of \eqref{fn} is 
$$f_m(t)=c_3 \sin\left(\frac{m\pi}{T}t\right),\quad m\in\n,$$
where $\delta=\frac{m\pi}{T}$. Notice that $m\not=0$. Definitively, the eigenfunctions are $u_{m,n}(s,t)=f_m(t)g_n(\theta)$ and the eigenvalues are
\begin{equation}\label{ekc}
 \lambda_{m,n}=\frac{m^2\pi^2}{(1+r^2)T^2}+\frac{n^2-1+n^2r^2}{r^2(1+r^2)}.
\end{equation}
Moreover we have $\lambda_{m,n}>\lambda_{m,0}$ and $\lambda_{m,n}>0$ if $n\geq 1$, and the sequence $\{\lambda_{m,0}\}_{m>0}$ is strictly increasing in $m$. Therefore, in order to compute the index of $C_{R}^T$ we only consider the eigenvalues $\lambda_{m,0}$. Furthermore, the index of $C_{R}^T$ is $m$ whenever $\lambda_{m,0}<0$ and $\lambda_{m+1,0}\geq0$. For that matter, by defining $\eta(T)$ as in \eqref{eqindex} we conclude $\mathrm{index}(C_{R}^T)=\eta(T)$. 

Finally,  if $T>\pi\sinh R$, then $\mathrm{index}(C_{R}^T)=1$ and  $C_{R}^T$ is not strongly stable. Similarly, if $T>2\pi\sinh R$ then $C_{R}^T$ is not stable because   $\mathrm{index}(C_{R}^T)=2$. 

The last statement is consequence of the above two items and the definition of the index of an unbounded surface given in \eqref{stable-infinito}. For this, we consider an exhaustion of $C_R$ given by the subdomains $C_R^T$ with $T\to\infty$.
\end{proof}

The next objective is to establish an analogy to the Plateau-Rayleigh result for Killing cylinders in the partitioning problem when the support surface is a pair of totally geodesic planes or a pair of horospheres. Consider again   $C_R^T$. The boundary of $C_R^T$ is formed by the  circles   $\partial C_R^T=\Gamma_R\cup\Gamma_R'$, where $\Gamma_R'=C_R^T\cap \{(x,y,z)\in\r^3_{+}:x^2+y^2=r^2\tau^2\}$. Notice that $\partial C_R^T$ is included in two   totally geodesic planes and also in two   horospheres, namely, 
$$\Gamma_R\subset\H_1\cap\S_{\tau_0},\quad \tau_0=\sqrt{1+r^2},$$
$$\Gamma_R'\subset\H_{\tau_1}\cap\S_{\tau_2},\quad \tau_1=e^T, \tau_2=\tau_0 e^T.$$
We see   $C_R^T$ as a capillary surface supported on either $\mathsf{S}_{\tau_0}\cup\mathsf{S}_{\tau_2}$ or $\mathsf{H}_1\cup\mathsf{H}_{\tau_1}$. However, we notice that if $T$ is sufficiently big, then $C_R^T$ is also a capillary surface supported on  $\mathsf{S}_{\tau_0}\cup \mathsf{H}_{\tau_1}$ or $\mathsf{H}_1\cup \mathsf{S}_{\tau_2}$.

The domain $W$ is bounded by   $\S_{\tau_0}$ and $\S_{\tau_2}$ in the first case, and $W$ is the slab between $\H_1$ and $\H_{\tau_1}$ in the second one. In both situations, the domain $D$ is the inner region in $W$ bounded by $C_R^T$. See Fig. \ref{fig3}.

\begin{figure}[h]
\begin{center}
\includegraphics[width=.4\textwidth]{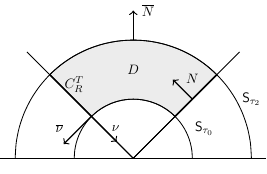}\quad \includegraphics[width=.4\textwidth]{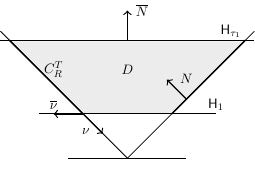}
\end{center}
\caption{Left: a Killing cylinder supported in two totally geodesic spheres. Right: a Killing cylinder supported in two horospheres.}\label{fig3}
\end{figure}

\begin{theorem}\label{t4} Let   $R,T>0$ and $\eta(T)$ as in \eqref{eqindex}. Consider $C_R^T$ as a capillary surface supported in either $\mathsf{S}_{\tau_0}\cup\mathsf{S}_{\tau_2}$ or $\mathsf{H}_1\cup\mathsf{H}_{\tau_1}$. Then, 
$$
\mathrm{index}(C_R^T)=\left\lbrace
\begin{array}{lll}
\eta(T)  & \mathrm{in} & \mathsf{S}_{\tau_0}\cup\mathsf{S}_{\tau_2},\\
\eta(T)+1 & \mathrm{in} & \mathsf{H}_1\cup\mathsf{H}_{\tau_1}.
\end{array}
\right.
$$
In consequence, if $C_R^T$ is supported between geodesic spheres, then the following holds:
\begin{enumerate}
\item If $T>\pi\sinh R$, then $C_R^T$ is not strongly stable.
\item If $T>2\pi\sinh R$, then $C_R^T$ is not stable.
\end{enumerate}
If $C_R^T$ is supported between horospheres, then $C_R^T$ is not strongly stable regardless of $T$, and if $T>\pi\sinh R$ then is not stable. 
 In particular, in both cases, $C_R$ is not stable. 
\end{theorem}

\begin{proof}

We begin by assuming that $C_R^T$ is  a capillary surface in $\S_{\tau_0}\cup \S_{\tau_2}$. It is clear that $\gamma=\pi/2$. Since $\overline{A}=0$, then $\q=0$. Let use   separation of variables, $u(t,\theta)=f(t)g(\theta)$. The conormal vector $\nu$ is $\pm\Psi_t/\sqrt{1+r^2}$. Thus $\frac{\partial u}{\partial\nu}=0$ is equivalent to $f'(0)=f'(T)=0$ and the eigenvalue problem is
 \begin{equation}\label{eqd3}
\left\lbrace\begin{array}{ll}
\frac{1}{1+r^2}f''g+\frac{1}{r^2}fg''+(\varpi+\lambda)fg=0 & \mathrm{in}\ [0,T]\times[0,2\pi],\\
f'(0)=f'(T)=0. &  
\end{array}\right.
\end{equation} 
The computations and the discussion of the cases are similar as in Thm. \ref{t3}.  We have $g_n$ as in \eqref{gn} and $f$ satisfies \eqref{fn}.
\begin{enumerate}
\item Case $ \varpi-\frac{n^2}{r^2}+\lambda =0$. Then  $f(t)=1$, the eigenfunctions are $u_{n}(t,\theta)=g_n(t)$ and the eigenvalues are 
$$\lambda_n=\frac{n^2}{r^2}-\frac{1}{r^2(1+r^2)}.$$
Note that the case $n=0$ would yield $u$ a constant function, which does not fulfill \eqref{eqd3}. Hence $n\geq1$ and all the eigenvalues are positive.
\item Case $(1+r^2)(\varpi-\frac{n^2}{r^2}+\lambda)=-\delta^2>0$. Then $f(t)=c_1\cosh(\delta t)+c_2\sinh(\delta t)$. The boundary condition $f'(0)=0$ yields $c_2=0$, and $f'(T)=0$ yields $\delta=0$, a contradiction. Hence this case is not possible.
 \item Case $(1+r^2)(\varpi-\frac{n^2}{r^2}+\lambda)=\delta^2>0$. Then  $f_m(t)= \cos(m\pi t/T)$, $m\in\n$, $\delta=m\pi/T$ and 
\[u_{m,n}(t,\theta)=\cos\left(\frac{ m \pi}{ T} t\right)(c_1 \cos( n\theta )+c_2 \sin( n\theta )).\]
The eigenvalues are given in \eqref{ekc}, hence the negative ones appear when $\lambda_{m,0}<0$. Thus the final argument carries over verbatim as the proof of Thm. \ref{t3}.
\end{enumerate} 

We now study $C_R^T$ as a capillary surface in $\H_1\cup\H_{\tau_1}$.  We have $N$ as in \eqref{n} and $\overline{N}_1=-\e_3$, $\overline{N}_{\tau_1}=\tau_1 \e_3$. The contact angle $\gamma_1\in (\pi/2,\pi)$ along $\Gamma_R$ is $\cos\gamma_1=-r/\sqrt{1+r^2}$, and along $\Gamma_R'$, the angle is $\gamma_2\in (0,\pi/2)$, with $\cos\gamma_2=r/\sqrt{1+r^2}$. The outward conormal vectors of $\Sigma$ along $\Gamma_R\cup\Gamma_R'$ are
$$\nu_1=-\frac{1}{\sqrt{1+r^2}}(r\cos\theta,r\sin\theta,1),\quad \nu_{\tau_1}=\frac{\tau_1}{\sqrt{1+r^2}}(r\cos\theta,r\sin\theta,1).$$
The outward conormal vectors of $\H_1$ and $\H_\tau$ along $\partial\Sigma$ are
$$\overline{\nu}_1=(\cos\theta,\sin\theta,0),\quad\overline{\nu}_{\tau_1}=\tau_1(\cos\theta,\sin\theta,0).$$
We compute $A(\nu,\nu)$. Using \eqref{r} and \eqref{n}, we get 
\begin{equation}\label{eqAcilindro}
A(\nu,\nu)=-\langle\nabla^e_\nu N,\nu\rangle+\frac{N_3}{\tau}=\frac{N_3}{\tau}=\frac{r}{\sqrt{1+r^2}}.
\end{equation}
On the other hand, bearing in mind that $\overline{A}$ is computed with respect to $-\overline{N}$ one has
$\overline{A}(\overline{\nu}_1,\overline{\nu}_1)=1=-\overline{A}(\overline{\nu}_{\tau_1},\overline{\nu}_{\tau_1})$. Thus
\begin{align*}
\q_1&=\frac{1}{\sin\gamma_1}+\cot\gamma_1 \frac{r}{\sqrt{1+r^2}}=\frac{1}{\sqrt{1+r^2}},\\
\q_{\tau_1}&=-\frac{1}{\sin\gamma_2}+\cot\gamma_2 \frac{r}{\sqrt{1+r^2}}=-\frac{1}{\sqrt{1+r^2}}.
\end{align*}
Letting $u(t,\theta)=f(t)g(\theta)$, we have 
$$\frac{\partial u}{\partial\nu_1}=  -\frac{f'g}{\sqrt{1+r^2}}=-\frac{\partial u}{\partial\nu_{\tau_1}}.$$
 The eigenvalue problem is as in \eqref{eqd3} but with boundary conditions
\begin{equation*}
\left\lbrace\begin{array}{ll}
f'(0)+f(0)=0,\\
f'(T)+f(T)=0. &
\end{array}\right.
\end{equation*}
Again, we have $g_n$ as in \eqref{gn} and $f$ satisfies \eqref{fn}.  
\begin{enumerate}
\item Case $ \varpi-\frac{n^2}{r^2}+\lambda =0$. Then $f(t)=a+bt$. The boundary conditions imply $f=0$. This case is not possible.
\item Case $(1+r^2)(\varpi-\frac{n^2}{r^2}+\lambda)=\delta^2>0$. Then $f(t)=c_3\cos(\delta t)+c_4\sin(\delta t)$. The first boundary condition implies $c_4\delta+c_3=0$. Thus $f(t)=-c_4\delta\cos(\delta t)+c_4\sin(\delta t)$. The second boundary condition gives $\sin(\delta T)=0$ and thus, 
$\delta T=m\pi$, $m\in\n$.   Then 
\[u_{m,n}(t,\theta)=\left(-\frac{m\pi}{T} \cos\left(\frac{m\pi}{T} t\right)+\sin\left(\frac{m\pi}{T} t\right)\right)g_n(\theta) \]
and the eigenvalues are as in \eqref{ekc}. In consequence,  the instability argument holds as in the case of  totally geodesic planes.  

\item   Case $(1+r^2)(\varpi-\frac{n^2}{r^2}+\lambda)=-\delta^2<0$. Then $f(t)=c_3\cosh(\delta t)+c_4\sinh(\delta t)$. The first boundary condition implies $c_4\delta+c_3=0$. Then    $f(t)=-c_4\delta\cosh(\delta t)+c_4\sinh(\delta t)$. The second boundary condition implies $\delta=1$. Thus the eigenfunctions are $u_{n}(t,\theta)=-e^{-t}g_n(\theta)$  and the eigenvalues are
$$\lambda_{n}=\frac{n^2-1}{r^2}.$$
Then $\lambda_n=\lambda_{0,n}$ and there is only one negative eigenvalue, namely $\lambda_0=-1/r^2$.

\end{enumerate}
 The last statement is consequence of \eqref{stable-infinito}.

\end{proof}

A slightly variation of Thm. \ref{t4} is the following. Consider the Killing cylinder $C_R$. The totally geodesic plane  $\P_0$ separates $C_R$ into two Killing half cylinders. Let $C_R^{+}=C_R\cap\{(x,y,z)\in\r^3_{+}:y>0\}$. The surface $C_R^{+}$ is a capillary surface on $\P_0$ with contact angle $\gamma=\pi/2$, hence its stability is studied by taking compact pieces of $C_R^{+}$. Indeed, for $T>0$, let $C_R^{+,T}=C_R^{+}\cap C_R^T$ and consider the eigenvalue problem \eqref{eq1} with the additional assumption that the test function $u$ vanishes along $\partial C_R^{+,T}\cap\partial C_R^T$. After a symmetry about $\P_0$, the Morse index agrees with the one of $C_R^T$ in Thm. \ref{t4}. This gives the same estimate $T>2\pi\sinh R$ for instability. 
 
In contrast, if we assume Dirichlet boundary conditions along $\partial C_R^{+,T}$, then $C_R^{+,T}$ is strongly stable. The proof of this property can be done in two ways. First, considering the corresponding eigenvalue problem. Following the notation used in this section,   we have   $C_R^{+,T}= \Psi([0,T]\times [0,\pi])$. By separation of variables $u(t,\theta)=f(t)g(\theta)$, the function $f$ satisfies \eqref{fn} with boundary conditions $  f(0)= f(T)=0$ and $ g(0)=g(\pi)=0$. Then the eigenfunctions are $u_{m,n}(t,\theta)=\sin(\frac{m\pi}{T} t) \sin(n\theta)$, $n,m\geq 1$ and the eigenvalues are $  \lambda_{m,n}$ as in \eqref{ekc}. Since  $m,n\geq 1$, we have that all eigenvalues are positive, proving the strongly stability of  $C_R^{+,T}$ for any $T>0$. A second argument is to observe that $C_R^+$ is a horocycle graph on a domain of $\P_0$. Thus the function $h=\langle N,\e_2\rangle$ is a function with constant sign on $C_R^+$. Since  $h$ is a Jacobi function,  $Jh=0$, 
then strong stability is immediate  \cite{fs1}.

 
\section{Index of a cylinder in a ball}\label{sb}

Wang and Chia proved that any stable   capillary  surface in a ball  of $\h^3$ is totally umbilical  \cite{wx}. Killing cylinders included in a ball are  capillary surfaces but not umbilical. Thus it is natural to ask about its index.    We point out that the computation of the index of particular capillary surfaces in a ball is not easy. For example,    it has been recently computed the index of the critical catenoid in a ball \cite{dev,sz,tr}. For a cylinder in a Euclidean ball, the study has been addressed in \cite{tz2}.

Suppose that $\Sigma$ is a piece of a Killing cylinder of radius $R$ included in a ball of radius $\rho>0$, $r=\sinh R$. Without loss of generality, we suppose that this ball is the interior domain of the sphere $\S(c,\rho)=\{(x,y,z)\in\r^3_+:x^2+y^2+(z-c)^2=\rho^2\}$, where $c>\rho$. For saving notation, we will simply denote the sphere by $\S$. The mean curvature of $\S$ is $H_0=c/\rho$. In order that $\Sigma$ intersects $\S$, the radius $r$ must satisfy
$$
r<\frac{\rho}{\sqrt{c^2-\rho^2}}=\frac{1}{\sqrt{H_0^2-1}}.
$$
When $r=\frac{1}{\sqrt{H_0^2-1}}$,   the cylinder is tangent to $\S$ and if $r>\frac{1}{\sqrt{H_0^2-1}}$ the cylinder does not intersect $\S$. Using the parametrization \eqref{p}, the Killing cylinder $\Sigma$  is   $C_r([t_{-},t_{+}]):=\Psi([t_{-},t_{+}]\times\r)$ (see Fig. \ref{figpartitioningball}),  where $t_{\pm}$ are determined by 
$$
e^{t_{\pm}}=\frac{1}{1+r^2}(c\pm\sqrt{\rho^2-r^2(c^2-\rho^2)})=\frac{\rho}{1+r^2}(H_0\pm\sqrt{1-r^2(H_0^2-1)}).
$$

\begin{figure}[h]
\begin{center}
\includegraphics[width=.55\textwidth]{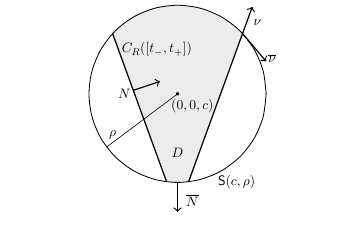}
\caption{The partitioning problem when the Killing cylinder is inside the ball determined by $\S(c,\rho)$.}
\label{figpartitioningball}
\end{center}
\end{figure}

\begin{theorem}\label{t-ib}
 Consider the Killing cylinder $C_r([t_{-},t_{+}])$ included in the ball determined by $\S$. Then, $\mathrm{index}(C_r([t_{-},t_{+}]))\geq 1$. Moreover,
\begin{enumerate}
\item  The index grows to $\infty$ as $r\to 0$.
\item There exists $r_0$ close to the value $\frac{1}{\sqrt{H_0^2-1}}$ such that $\mathrm{index}(C_r([t_{-},t_{+}]))= 1$  for all $r\in (r_0,\frac{1}{\sqrt{H_0^2-1}})$.
\item The function $r\mapsto \mathrm{index}(C_r([t_{-},t_{+}]))$ is decreasing in the interval $(0,\frac{1}{\sqrt{H_0^2-1}})$ varying from $\infty$ to $1$.
\end{enumerate}
\end{theorem}

\begin{proof}
In order to simplify the notation, let  
$$\alpha=\sqrt{1-r^2(H_0^2-1)},\quad \beta=\frac{\rho}{1+r^2}.$$
In particular, 
$$
e^{t_{\pm}}=\beta(H_0\pm\alpha).
$$
The unit normals $N_e,\overline{N_e}$, of $C_R$ and $\mathsf{S}$ respectively, are given by
$$
N_e=\frac{1}{\sqrt{1+r^2}}(-\cos\theta,-\sin\theta,r),\quad \overline{N_e}(x,y,z)=\frac{1}{\rho}(x,y,z-c).
$$
The contact angle $\gamma$ between $C_r([t_{-},t_{+}])$ and $\S$ is determined by 
$$
\cos\gamma=-\frac{rH_0}{\sqrt{1+r^2}},\quad\sin\gamma=\frac{\sqrt{\rho^2-r^2(c^2-\rho^2)}}{\rho\sqrt{1+r^2}}=\frac{\alpha}{\sqrt{1+r^2}}.
$$
Note that $\overline{A}(\overline{\nu},\overline{\nu})=H_0$, since $\mathsf{S}$ is totally umbilical, while it was exhibited in \eqref{eqAcilindro} that $A(\nu,\nu)=\frac{r}{\sqrt{1+r^2}}$. Then,
$$\q=\frac{H_0\sqrt{1+r^2}}{\alpha}-\frac{H_0r^2}{\alpha\sqrt{1+r^2}}=\frac{H_0}{\alpha\sqrt{1+r^2}}.$$
 Let us take $u(t,\theta)=f(t)g(\theta)$. Recall that $\frac{\partial u}{\partial\nu}(t_-)=-f'g/\sqrt{1+r^2}$ and  $\frac{\partial u}{\partial\nu}(t_+)=f'g/\sqrt{1+r^2}$.
 Then the Robin conditions are equivalent to 
\begin{align*}
f'(t_-)+\displaystyle{\frac{H_0}{\alpha}}f(t_-)=0,\\
f'(t_+)-\displaystyle{\frac{H_0}{\alpha}}f(t_+)=0.
\end{align*}
Consequently, the eigenvalue problem is
$$
\left\lbrace
\begin{array}{ll}
\vspace{.15cm}\frac{1}{1+r^2}f''g+\frac{1}{r^2}fg''+(\varpi+\lambda)fg=0 & \mathrm{in}\ [t_-,t_+]\times[0,2\pi],\\
\vspace{.15cm}f'(t_-)+\frac{H_0}{\alpha}f(t_-)=0,\\
f'(t_+)-\frac{H_0}{\alpha}f(t_+)=0.
\end{array}\right.
$$
Arguing as in the proof of Thm. \ref{t3} we have $g_n$ as in \eqref{gn} and $f$ satisfies \eqref{fn}.   We discuss cases.
\begin{enumerate}
\item Case $\varpi-\frac{n^2}{r^2}+\lambda=0$. Then, $f(t)=at+b$. The above boundary conditions yield
\begin{equation*}
a+\frac{H_0}{\alpha}(at_-+b)=0,\quad a-\frac{H_0}{\alpha}(at_++b)=0.
\end{equation*}
If there are non-trivial solutions  $a$ and $b$, then
$$2+\frac{H_0}{\alpha}(t_{-}-t_{+})=0.$$
This is equivalent to 
\begin{equation}\label{aa}
0=2+\frac{H_0}{\alpha}\log\frac{H_0-\alpha}{H_0+\alpha}.
\end{equation}
Let $\sigma=H_0/\alpha$. Since $\alpha$ varies from $0$ to $1$, then $\sigma$ varies from $H_0$ to $\infty$. 
The function 
$$\sigma\mapsto 2+\sigma\log\frac{\sigma-1}{\sigma+1}$$
is strictly increasing on $\sigma$ going to $0$ as $\sigma\to\infty$.  This proves that this case is not possible.
\item Case $(1+r^2)(\varpi-\frac{n^2}{r^2}+\lambda)=-\delta^2$ with $\delta>0$. Then, $f(t)=ae^{\delta t}+be^{-\delta t}$. 
The existence of non-trivial constants $a,b$ fulfilling the boundary conditions is equivalent to
$$
\left(\frac{H_0+\alpha}{H_0-\alpha}\right)^\delta\left(\frac{H_0-\alpha\delta}{H_0+\alpha\delta}\right)=1.
$$
Let $h(\delta)$ be  the left-hand side of the previous equation. We are looking for   a positive solution $\delta$ of the equation $h(\delta)=1$. The function $h$ satisfies $h(0)=1$, then $h$ is increasing for $\delta>0$, achieving a global maximum, and then decreases towards $-\infty$ as  $\delta\rightarrow\infty$. Consequently $h(\delta)=1$ has a unique solution. Since $h(1)=1$, then  $\delta_o=1$ is the solution. The eigenvalues are
$$
\lambda_n=-\frac{\delta_o^2}{1+r^2}+\frac{n^2(1+r^2)-1}{r^2(1+r^2)}=\frac{n^2-1}{r^2}.$$
Hence $n=0$ gives the only negative eigenvalue, contributing in $1$ to the index of $\Sigma$. This proves the first statement of the theorem.
\item Case $(1+r^2)(\varpi-\frac{n^2}{r^2}+\lambda)=\delta^2$, $\delta>0$. Then, $f(t)=a\cos(\delta t)+b\sin(\delta t)$.  Arguing as in the previous case, there exist non-trivial values $a,b$ fulfilling the Robin boundary condition if and only if 
$$
2H_0\alpha\delta\cos((t_{+}-t_{-})\delta)-(H_0^2-\alpha^2\delta^2)\sin ((t_{+}-t_{-})\delta)=0.
$$

If  $T=t_{+}-t_{-}=\log\frac{H_0+\alpha}{H_0-\alpha}$, then the above equation can be expressed as

$$
\tan (\delta T)=\frac{2H_0\alpha\delta}{H_0^2-\alpha^2\delta^2} =\frac{2\sigma\delta}{\sigma^2-\delta^2},\quad\sigma=\frac{H_0}{\alpha}. 
$$
Consider the functions $p(\delta)=\tan(\delta T)$ and $q(\delta)=2\sigma\delta/(\sigma^2-\delta^2)$. Then $q(0)=0$, $q$ strictly increases and has a vertical asymptote at $\delta=\sigma$. For $\delta>\sigma$, $q$ is also strictly increasing and negative and converges to zero as $\delta\rightarrow\infty$.

We now introduce the following notation, which will be used in the rest of the paper.  
\begin{equation}\label{interval}
I_0=\left(0,\frac{\pi}{2T}\right),\quad I_m=\left(\frac{(2m-1)\pi}{2T},\frac{(2m+1)\pi}{2T}\right), m\geq 1.
\end{equation}
The equation $p(\delta)=q(\delta)$ has  at most one solution in each interval $I_m$, $m\geq1$, obtained by the intersection of the negative branch of $q(\delta)$ with $p(\delta)$. Let $\delta_m$ be such a solution if it exists. 
See Fig. \ref{fig-pq}.

\begin{figure}[h]
\begin{center}
 \includegraphics[width=.6\textwidth]{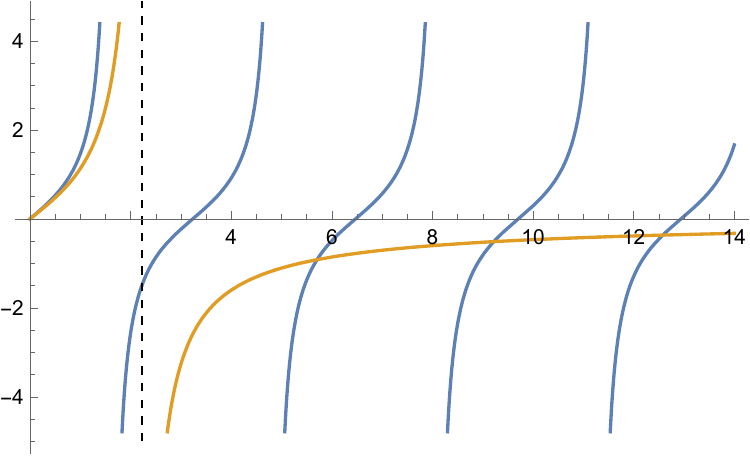} 
\end{center}
\caption{Solutions  of the equation $p(\delta)=q(\delta)$. The function $p=p(\delta)$ (blue) have vertical asymptotes at $\delta=\frac{(2m+1)\pi}{2T}$, $m\in\n$, and the function $q=q(\delta)$ (orange), only one at $\delta=\sigma$. This vertical asymptote appears in the figure as  the dotted line.}\label{fig-pq}
\end{figure}

Consequently, the solutions $\delta_m$ form a discrete set of positive numbers and $\delta_m\to\infty$. On the other hand, the eigenvalues are
$$
\lambda_{m,n}=\frac{\delta_m^2}{1+r^2}+\frac{n^2}{r^2}-\frac{1}{r^2(1+r^2)}.
$$
Hence, the negative ones appear only if $n=0$. In such a case,  we have $\lambda_{m,0}<0$ if and only if  $\delta_m<\frac{1}{r}$. 

{\it Claim}. There is no intersection between the graphic of $p(\delta)$ and the positive branch of $h(\delta)$. See Fig. \ref{fig-pq}.

To prove the claim, first we prove that   $\frac{\pi}{2T}<\sigma$. This inequality is equivalent to  $e^\pi<\left(\frac{\sigma+1}{\sigma-1}\right)^{2\sigma}$ for all $\sigma>1$. However the function in the right hand side of this inequality decreases until the value $e^4$. Analogously, it can be proved that $\sigma<\frac{\pi}{T}$, proving the claim.

As a consequence, if there is an intersection between $p(\delta)$ and the positive branch of $h(\delta)$, this must occur in the interval $(0,\frac{\pi}{2T})$, that is, with the first branch of $p(\delta)$. However we have  $p'(\delta)>q'(\delta)$ in $(0,\frac{\pi}{2T})$. Since $p(0)=q(0)$, this shows that $p>q$ in $(0,\frac{\pi}{2T})$.

We now prove the three items of the theorem.
\begin{enumerate}
\item   If $r\to 0$, then $\alpha\to 1$. This implies that the vertical asymptote $x=\sigma$ of $q$ moves to the left until the vertical line $x=H_0$ in the limit. Thus the negative branch of $q(\delta)$ intersects infinite times with $p(\delta)$. Since $\frac{1}{r}\to\infty$, the number of values $\delta_m$ satisfying   $\delta_m<\frac{1}{r}$  increases to $\infty$ as $r\to 0$.  
\item    If $r\to 1/\sqrt{H_0^2-1}$, then $\alpha\to 0$, hence $T\to 0$. In particular, there is $r_0>0$ such that if $r\in (r_0,\frac{1}{\sqrt{H_0^2-1}})$, then 
$$\sqrt{H_0^2-1}<\frac{1}{r}<\frac{\pi}{2T}.$$
Since the negative eigenvalues $\lambda_{m,0}$ are given by the solutions $\delta_m$ with $\delta_m<1/r$, then this solution would be between the positive branch of $q$ with the first branch of $p$. But this is not possible by the claim.
\item The set $\{\delta_{m}:m\in\n\}$ is   discrete  with $\delta_m\to\infty$. Since the negative eigenvalues correspond when   $\delta_m<1/r$, this number increases as $r$ decreases.  
\end{enumerate}

\end{enumerate}

 \end{proof}
 
 From (1)   of Thm. \ref{t-ib}, the Killing cylinder $C_r([t_{-},t_{+}])$ is not stable if $r$ is sufficiently close to $0$. In case that the index is $1$, we know from   \eqref{ii} that the weak index is less or equal than $1$. In fact,  Killing cylinders are not stable: see \cite[Thm. 5.3]{so}.

\section{Stability of  Killing cylinders contained in unbounded domains determined by horospheres, geodesic planes and equidistant surfaces}\label{s5}

   In this section, we consider unbounded pieces of Killing cylinders supported in horospheres, geodesic planes and equidistant surfaces. First, we begin when the support is a horosphere. In the upper halfspace model of $\h^3$, an after an isometry, a horosphere is a horizontal plane which determines two non-isometric domains in $\h^3$.   Without loss of generality, we assume that the horosphere is   $\H_1$ and   $W$ is the domain $z>1$ over this plane. Notice that the intersection of the boundary of $W$ with the ideal boundary of $\h^3$ is one point.

 Consider the Killing cylinder $C_R$. Then  $\H_1$ separates $C_R$ in two pieces and let $C_{R,+}=C_R\cap W$.  Since  $C_{R,+}$ is not compact, we consider an exhaustion of $C_{R,+}$ by cutting it   in pieces of length $T$ in the $t$-variable and letting $T\to\infty$. So, let $C_{R,+}^{T}=\Psi([0,T]\times [0,2\pi])$. See Fig. \ref{fig6}.
  
\begin{figure}[h]
\begin{center}
\includegraphics[width=.4\textwidth]{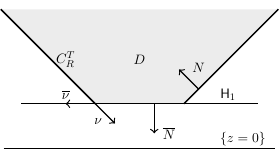}
\caption{The partitioning problem when the boundary of the Killing cylinder is included in a horosphere.}
\label{fig6}
\end{center}
\end{figure}

\begin{theorem}\label{t6} Let   $R,T>0$ and define
$$
\eta(T)=\max\{m\in\n\cup\{0\}: \delta_m<\frac{1}{\sinh R}\},
$$
where $\delta_m$ is the only solution of the equation $x=\tan(Tx)$ in each interval $I_m$ defined in \eqref{interval}, $m\geq 0$. Here $\delta_0$ is the root of the equation  in the interval $(0,\frac{\pi}{2T})$ in case that  $T>1$.    
Then   
$$\mathrm{index}(C_{R,+}^T)=\left\{\begin{array}{ll} \eta(T)& T< 1,\\
1+ \eta(T)& T\geq 1.
\end{array}\right.$$
In consequence, the Killing cylinder $C_{R,+}$ is not stable.
\end{theorem} 

\begin{proof} As usual, let $r=\sinh R$. We have $N$ given in \eqref{n} and $\overline{N}=-\e_3$. The contact angle $\gamma\in (\pi/2,\pi)$ satisfies  $\cos\gamma=-r/\sqrt{1+r^2}$ and $\sin\gamma=1/\sqrt{1+r^2}$. The conormal vectors along $\partial\Sigma$ are
$$\nu=-\frac{1}{\sqrt{1+r^2}}(r\cos\theta,r\sin\theta,1),\quad\overline{\nu}=(\cos\theta,\sin\theta,0).$$
We compute $A(\nu,\nu)$. Using \eqref{r} and \eqref{n}, we get 
$$A(\nu,\nu)=-\langle\nabla^e_\nu N,\nu\rangle+\frac{N_3}{z}=\frac{N_3}{z}=\frac{r}{\sqrt{1+r^2}}.$$
Since $\overline{A}(\overline{\nu},\overline{\nu})=1$, we have
$$\q=\frac{1}{\sin\gamma}+\cot\gamma \frac{r}{\sqrt{1+r^2}}=\frac{1}{\sqrt{1+r^2}}.$$
 Using again separation of variables, $u(t,\theta)=f(t)g(\theta)$ and since $\frac{\partial u}{\partial\nu}= -f'g/\sqrt{1+r^2}$, the eigenvalue problem is   as \eqref{eqd3} but with  boundary conditions 

\begin{equation*}
\left\lbrace\begin{array}{ll}
f'(0)+ f(0)=0, &\\
f(T)=0.&
\end{array}\right.
\end{equation*}
We known that $g$ is $g_n$ as in \eqref{gn} and that $f$ satisfies \eqref{fn}.  
\begin{enumerate}
\item Case $ \varpi-\frac{n^2}{r^2}+\lambda =0$. Then $f(t)=a+bt$. The boundary conditions imply that there is a non-trivial $f$ if and only if $T=1$. The eigenvalues are  $ \frac{n^2}{r^2}-\varpi$ and only negative eigenvalues   appear if $n=0$. Thus if $T=1$, there is one in the contribution of the Morse index; otherwise, this case is not possible.  
\item Case $(1+r^2)(\varpi-\frac{n^2}{r^2}+\lambda)=-\delta^2<0$, $\delta>0$. Then $f(t)=c_3\cosh(\delta t)+c_4\sinh(\delta t)$. The boundary conditions imply 
$$\delta c_4+c_3=0,\quad \cosh (\delta T)c_3+\sinh(\delta T)c_4=0.$$
Thus  $f(t)=-\delta\cosh(\delta t)+\sinh(\delta t)$, where $\tanh(\delta T)=\delta$. This equation has a unique nonzero solution $\delta_o$ if and only $T>1$. Moreover $\delta_o<1$.  So, if   $T>1$  the eigenfunctions are
$$
u_n(t,\theta)=(-\delta\cosh(\delta t)+\sinh(\delta t))g_n(\theta),
$$
and the eigenvalues are  
$$
\lambda_n=-\frac{\delta_o^2r^2+1}{r^2(1+r^2)}+\frac{n^2}{r^2}.
$$
Note that $\lambda_0<0$ and $\lambda_n>0$ for every $n\geq 1$ because $\delta_o<1$. Hence, the only negative eigenvalue is $\lambda_0$ ($n=0$) which contributes with $1$ in the index if and only if $T>1$.

\item Case $(1+r^2)(\varpi-\frac{n^2}{r^2}+\lambda)=\delta^2>0$, $\delta>0$. Then $f(t)=c_3\cos(\delta t)+c_4\sin(\delta t)$. The boundary conditions imply 
$$c_3+\delta c_4=0,\quad \cos (\delta T)c_3 +\sin(\delta T)c_4=0.$$
Thus  $f(t)=-\delta\cos(\delta t)+\sin(\delta t)$, where $\tan(T\delta)=\delta$. For $m\geq 1$, in each interval $I_m$ of \eqref{interval}, there is only one root $\delta_m$ of   the equation $x=\tan (Tx)$. In the interval $(0,\frac{\pi}{2T})$ the existence of a root $\delta_0$ is assured if and only if $T<1$. In any caes, the eigenfunctions are
\[u_{m,n}(t,\theta)=(-\delta_m\cos(\delta_m t)+\sin(\delta_m t))g_n(\theta) \]
and the eigenvalues are
$$\lambda_{m,n}=\frac{\delta_m^2}{1+r^2}+\frac{n^2}{r^2}-\frac{1}{r^2(1+r^2)}.$$
For $n\geq1$ all the eigenvalues are positive, hence the negative ones appear when $n=0$. Therefore, the eigenvalues are negative if and only if  
 $ \delta_m<\frac{1}{r}$.

\end{enumerate} 

Now the result follows by recalling that $r=\sinh R$ and taking into account the contribution of negative eigenvalues of both cases depending if $T< 1$ or $T\geq1$. 

For the last statement, notice that as $T\to\infty$, the width of the intervals $I_m$ goes to $0$. Thus, once fixed $R$, the number of roots $\delta_m$ satisfying $\delta_m< 1/\sinh R$ goes to $\infty$ as $T\to 0$.

\end{proof}

The next umbilical surface is a totally geodesic plane. Viewed this surface as a hemisphere intersecting orthogonally  the plane $z=0$, the hemisphere determines two domains in $\r^3_+$: one bounded, the inside of the hemisphere and denoted by $W\subset\r^3_+$; and the other unbounded. In this case, both domains having the hemisphere as common boundary are isometric.

Consider the Killing cylinder  $C_R$ as in Thm. \ref{t6} but supported on the totally geodesic plane  $\S_{\tau_0}$, where $\tau_0=\sqrt{1+r^2}$. The domain $W$ will be the part above $\S_{\tau_0}$ and $D$ is the inside of $C_{R,+}$. Notice that this $D$ does not coincide with that of Thm. \ref{t6}. Let $C_{R,+}=C_R\cap W$, which it is the same as in Thm. \ref{t6} and let $C_{R,+}^T=C_{R,+}\cap C_R^T$. See Fig. \ref{fig7}. In the next result we establish the critical length that defines the instability of Plateau-Rayleigh type.
  
\begin{figure}[h]
\begin{center}
\includegraphics[width=.4\textwidth]{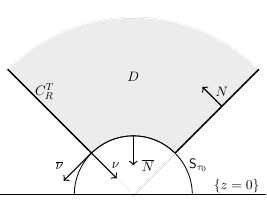}
\caption{The partitioning problem when the boundary of the Killing cylinder is included in a geodesic plane.}
\label{fig7}
\end{center}
\end{figure}

\begin{theorem}\label{t7} Let be $R,T>0$ and define
$$
\eta(T)=\max\{m\in\n\cup\{0\}: 2m+1<\frac{2T}{\pi\sinh R}\}.
$$
Consider $C_{R,+}^T$ as a capillary surface supported in $\S_{\tau_0}$. Then, $\mathrm{index}(C_{R,+}^T)=\eta(T)$. Moreover, if $T>3\pi\sinh R/2$, then $C_{R,+}^T$ is not stable. In consequence, $C_{R,+}$ is not stable.
\end{theorem} 

\begin{proof}
The arguments are similar as in Thm. \ref{t6} and we only give the main facts. Now $\gamma=\pi/2$, $ \overline{A}=0$ and $\q=0$.  By separation of variables,   the eigenvalue problem is as in the proof of Thm. \ref{t6} but the boundary conditions are $f'(0)=f(T)=0$.  The functions $g_n$ are as in \eqref{gn}  and $f=f_m$ appears in the case  $(1+r^2)(\varpi-\frac{n^2}{r^2}+\lambda)=\delta^2>0$. Then $u_{m,n}(t,\theta)=\cos(\frac{(2m+1)\pi}{2T} t)g_n(\theta)$, $m,n\in\n$,  and the eigenvalues are  
$$\lambda_{m,n}=\frac{(2m+1)^2\pi^2}{4T^2(1+r^2)}+\frac{n^2}{r^2}-\varpi.$$
If $n\geq 1$, then $\lambda_{m,n}>0$. Thus the negative eigenvalues appear if $n=0$.  In particular, as $T\to\infty$, there are many numbers $m$ such that $\lambda_{m,0}<0$, proving that $C_{R,+}$ is not stable. This is the last statement of the theorem.

The second eigenvalue $\lambda_{1,0}$ is negative if and only if  $T>3\pi\sinh R/2$, obtaining the critical length for the instability of Plateau-Rayleigh type. 

 \end{proof}

Finally, assume that the boundary of the Killing cylinder is included in an equidistant surface. For computational reasons, it is better to look the equidistant surfaces as Euclidean spherical caps. Thus, let us suppose that the equidistant surfaces $\E_{H_0}$ is the spherical cap in $\r^3_{+}$  of radius $\rho$ centered at $(0,0,-c)$, $\rho>c>0$, whose mean is $H_0=c/\rho$. We consider a Killing cylinder $C_R$ which we can assume that intersects $\E_{H_0}$ at $z=1$, hence $(1+c)^2+r^2=\rho^2$. See Fig. \ref{fig8}. We introduce the notation
$$
\Theta=\frac{H_0}{\sqrt{1+r^2(1-H_0^2)}}.
$$

\begin{figure}[h]
\begin{center}
\includegraphics[width=.4\textwidth]{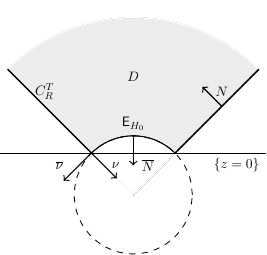}
\caption{The partitioning problem when the boundary of the Killing cylinder is included in an equidistant surface.}
\label{fig8}
\end{center}
\end{figure}

\begin{theorem}\label{t8}
Let be $R,T>0$ and define
$$
\eta(T)=\max\{m\in\n\cup\{0\}: \delta_m<\frac{1}{\sinh R}\},
$$
where $\delta_m$ is the only solution of the equation $x/\Theta=\tan(Tx)$ in each interval $I_m$, $m\geq 0$. Here $\delta_0$ is the root of the equation  in the interval $(0,\frac{\pi}{2T})$ in case that  $T\Theta>1$.    
Consider $C_{R,+}^T$ as a capillary surface supported in $\E_{H_0}$.  Then   
$$\mathrm{index}(C_{R,+}^T)=\left\{\begin{array}{ll} \eta(T)& T\Theta< 1,\\
1+ \eta(T)& T\Theta\geq 1.
\end{array}\right.$$
In consequence, the Killing cylinder $C_{R,+}$ is not stable.
\end{theorem} 

\begin{proof} First we compute the function $\q$.  Notice that $\overline{A}(\overline{\nu},\overline{\nu})=H_0$. The unit normals $N_e,\overline{N_e}$, of $C_R$ and $\mathsf{E}_{H_0}$  are 
$$
N_e=\frac{1}{\sqrt{1+r^2}}(-\cos\theta,-\sin\theta,r),\quad \overline{N_e}(x,y,z)=-\frac{1}{\rho}(x,y,z+c).
$$
The angle $\gamma$ between $N_e$ and $\overline{N_e}$ is determined by
$$\cos\gamma=-\frac{rH_0}{\sqrt{1+r^2}},\quad\sin\gamma=\frac{\sqrt{1+r^2(1-H_0^2)}}{\sqrt{1+r^2}},$$
hence
$$ \q=\frac{H_0}{\sin\gamma}-\frac{r^2H_0^2}{(1+r^2)\sin\gamma}=\frac{\Theta}{\sqrt{1+r^2}}.$$
 If $u(t,\theta)=f(t)g(\theta)$,   the boundary conditions are
\begin{equation*}
\left\lbrace\begin{array}{ll}
f'(0)+ \Theta  f(0)=0, &\\
f(T)=0.&
\end{array}\right.
\end{equation*}
The solutions $g_n$ are given in \eqref{gn}. On the other hand,  $f$ satisfies \eqref{fn}.  
\begin{enumerate}
\item Case $ \varpi-\frac{n^2}{r^2}+\lambda =0$. If $f(t)=a+bt$, then there are non-trivial $a,b\in\r$ if and only if $T=1/\Theta $. In such a case, 
the eigenvalues are $ \frac{n^2}{r^2}-\varpi$ and the negative eigenvalues  appear if $n=0$. Hence we deduce that if $T=1/\Theta $, there is  a contribution in one in the    Morse index. Otherwise, this case is not possible. 
\item Case $(1+r^2)(\varpi-\frac{n^2}{r^2}+\lambda)=-\delta^2<0$, $\delta>0$. Then $f(t)=c_3\cosh(\delta t)+c_4\sinh(\delta t)$ and the boundary conditions imply 
$$c_3 \Theta +\delta c_4=0,\quad \cosh (\delta T)c_3 +\sinh(\delta T)c_4=0.$$
Thus  $f(t)=-\frac{\delta}{\Theta }\cosh(\delta t)+\sinh(\delta t)$, where $\tanh(\delta T)=\delta/\Theta $. This equation has a unique nonzero solution $\delta_o$ if and only if $T\Theta >1$ and, in such a case, $\delta_o<1$. The eigenvalues are  
$$
\lambda_n=-\frac{\delta_o^2r^2+1}{r^2(1+r^2)}+\frac{n^2}{r^2}.
$$
Besides  $\lambda_0<0$, we have  $\lambda_n>0$ for every $n\geq 1$ because $\delta_o<1$.  Thus  the only negative eigenvalue is $\lambda_0$   which contributes in $1$ in the index if and only if $T\Theta>1$.

\item Case $(1+r^2)(\varpi-\frac{n^2}{r^2}+\lambda)=\delta^2>0$, $\delta>0$. Then $f(t)=c_3\cos(\delta t)+c_4\sin(\delta t)$. The boundary conditions imply 
$$\Theta  c_3+\delta c_4=0,\quad \cos (\delta T)c_3+\sin(\delta T)c_4=0.$$
Thus  $f(t)=-\frac{\delta}{\Theta }\cos(\delta t)+\sin(\delta t)$, where $\tan(T\delta)=\delta/\Theta $. In each interval $I_m$, $m\geq 1$,  there is only one root $\delta_m$ of   the equation $x/\Theta =\tan (Tx)$. In the interval $(0,\frac{\pi}{2T})$ the existence of a root $\delta_0$ is assured if and only if $T\Theta >1$. The eigenvalues are
$$\lambda_{m,n}=\frac{\delta_m^2}{1+r^2}+\frac{n^2}{r^2}-\frac{1}{r^2(1+r^2)}.$$
The negative eigenvalues appear if $n=0$. Hence, the eigenvalue is negative if and only if  
 $ \delta_m<\frac{1}{r}$.

\end{enumerate} 

Now the result follows as in Thm. \ref{t6}.  
 
\end{proof}
 
\section{The Delaunay surfaces obtained by bifurcating Killing cylinders}\label{s6}

In the hyperbolic space $\h^3$, a Delaunay surface is a rotational surface of elliptic type and with constant mean curvature $H$, with $|H|>1$  \cite{cd}. In the upper half space model of $\h^3$,   we can suppose that the rotation axis is the $z$-axis.  In this section we will use the Crandall-Rabinowitz  bifurcation theorem by simple eigenvalues to prove that the Delaunay surfaces can be obtained by  bifurcating Killing cylinders. In Euclidean space, it is known that a family of unduloids bifurcates off from the circular cylinder \cite{vo01}.  Here we follow the scheme  in \cite{ss} for the Euclidean case. Other results of bifurcation of rotational cmc surfaces in hyperbolic space have appeared in   \cite{jl,maz}.

Consider the Killing cylinder $C_R\subset\h^3$ of radius $R>0$ and let $H_0$ be its mean curvature. If $u$ is a smooth function defined on $C_R$, let $C_R(u)$ denote the normal graph on $C_R$ defined by     $C_R(u)=\{\mbox{exp}(p+u(p)N(p))\in\h^3:p\in C_R\}$. For convenience in the notation, we write  $p=\Psi(t,\theta)$, $\theta\in\s^1$ and we identify $u$ with $u\circ\Psi$. 
 Let    $X=C^\infty(\r/2\pi\z\times\s^1)$ be  the space  of smooth functions $u=u(t,\theta)$ defined in $\r\times\s^1$ that are $2\pi$-periodic in the   coordinate $t\in\r$.  

For each $T>0$ and $u\in X$, define   $\widetilde{u}_T(t,\theta)=u(\frac{2\pi}{T}t,\theta)$ which is $T$-periodic in the first variable. If $u>-\sinh R$ then $C_R(\widetilde{u}_T)$ is a regular surface parametrized by $(t,\theta)\in\r\times \s^1$. Let $H_T$ be the mean curvature of $C_R(\widetilde{u}_T)$. 

Let $V\subset X$ be an open set such that  $0\in V$ and $C_R(\widetilde{u}_T)$ is a normal graph on $C_R$ for all $u\in V$. Define
\begin{equation}\label{ff}
\mathcal{F}:V \times\r \rightarrow  X  ,\quad  \mathcal{F}(u,T)(t,\theta)=2\left(H_0-H_T\left(\frac{T}{2\pi}t,\theta\right)\right). 
\end{equation}
Notice that $\mathcal{F}(u,T)$ is $2\pi$-periodic in the second variable because $H_T$ is $T$-periodic on $t$. Consider the solutions of 
\begin{equation}\label{s}
\mathcal{F}(u,T)=0.
\end{equation}
Notice that $\mathcal{F}(0,T)=0$ for all $T$ because for   $\widetilde{u}_T=0$, the surface $C_R(\widetilde{u}_T)$ is the cylinder $C_R$ whose mean curvature is $H_0$. This gives a curve  $T\mapsto (0,T)$ of solutions of \eqref{s}. The purpose is to prove that there exists $T_0>0$ such that it is possible to perturb this curve  around a point $(0,T_0)$  obtaining    another curve  $s\mapsto (w_s,t(s))$ of solutions of \eqref{s} crossing the point $(0,T_0)$.  

By the   implicit function theorem, non-uniqueness  of \eqref{s} implies that $\mbox{Ker}( D_u \mathcal{F}(0,T))$ is not trivial. Notice that if $H(u)$ is the mean curvature of $C_R(u)$, then the   derivative of $(u,t)\mapsto 2(H_0-H(u))$ at $(0,T)$ is   nothing but the   Jacobi operator $J$ of $C_R$, that is, 
$$J=\Delta  + |A|^2-2 =\frac{1}{1+r^2}\partial_t^2+\frac{1}{r^2}\partial_{\theta}^2+\varpi.$$
 Using the chain rule and the definition of $\mathcal{F}$ in \eqref{ff}, we have  
\begin{equation}\label{l1}
L:= D_u\mathcal{F}(0,T)  =\left(\frac{2\pi}{T}\right)^2\frac{1}{1+r^2}\partial_t^2+\frac{1}{r^2}\partial_{\theta}^2+\varpi.
\end{equation}
In consequence, the kernel of $L$ is not trivial   if and only if   $0$ is an eigenvalue of $L$.   The result of  bifurcation of   Crandall and Rabinowitz   in its abstract version is the following \cite{cr}.

\begin{theorem}\label{t-bifu} Let $X$ and $Y$ be two real Banach spaces, $V$ an open subset of $0$ in $X$ and $\mathcal{F}:V\times I\rightarrow Y$ be a twice continuously Fr\'echet differentiable functional, where $I\subset \r$ is an open interval.    Let $t_0\in I$ and $u\in V$. Assume that $\mathcal{F}(0,t)=0$ for all $t\in I$ and the following conditions hold:  
\begin{enumerate}
\item $\mbox{dim Ker}(D_u  \mathcal{F}(0,t_0))=1$. Assume that Ker$(D_u \mathcal{F}(0,t_0))$ is spanned by $u_0$.
\item The codimension of  $\mbox{rank } D_u \mathcal{F}(0,t_0) $ is $1$.
\item $D_tD_u \mathcal{F}(0,t_0)(u_0)\not\in\mbox{rank }D_u \mathcal{F}(0,t_0) $.
\end{enumerate}
Then  there exists a   continuously differentiable curve $s\mapsto (w_s,t(s))$, $s\in(-\epsilon,\epsilon)$, with $w_0=0$, $t(0)=t_0$, such that $\mathcal{F}(w_s,t(s))=0$, for any $|s|<\epsilon$. Moreover, $(0,t_0)$ is a bifurcation point of the equation $\mathcal{F}(u,t)=0$ in the following sense: in a neighborhood  of $(0,t_0)$ in $V\times I$, the set of solutions of $\mathcal{F}(u,t)=0$ consists only of the curve $t\mapsto (0,t)$ and the curve $s\mapsto (w_s,t(s))$.
\end{theorem}

 With   $C^{2,\alpha}$ and $C^{0,\alpha}$, we denote the functions that have  H\"{o}lder continuous second derivatives or that are  H\"{o}lder continuous, respectively. In order to apply Thm. \ref{t-bifu}, let $X= C^{2,\alpha}( \r/2\pi\z \times\s^1)$ and $Y= C^{0,\alpha}( \r/2\pi\z \times\s^1)$. The first step is to calculate the eigenspace $\ee_0$ of the eigenvalue $\lambda=0$ for $L$. The eigenvalue problem to solve is  
\begin{equation}\label{eqb}
Lu+\lambda u=0\  \mathrm{ in }\ \r/2\pi\z\times\s^1.
\end{equation}
We expand $u$ by its Fourier expression 
$$u(t,\theta)=\sum_{n=0}^\infty f_n(t)\cos(n\theta)+\sum_{n=1}^\infty g_n(t)\sin(n\theta).$$
By    \eqref{l1}, the function $u$ satisfies   \eqref{eqb} if and only if
\begin{equation*}
\begin{split}
&\sum_{n=0}^\infty \left(f_n''(t)+\frac{T^2(1+r^2)}{4\pi^2}\left(\varpi+\lambda-\frac{n^2}{r^2} \right)f_n(t)\right)
\cos(n\theta)\\
&+\sum_{n=1}^\infty  \left(g_n''(t)+\frac{T^2(1+r^2)}{4\pi^2}\left(\varpi+\lambda-\frac{n^2}{r^2} \right)g_n(t)\right) \sin(n\theta)=0.
\end{split}
\end{equation*}
Thus 
$$f_n''(t)+\frac{T^2(1+r^2)}{4\pi^2}\left(\varpi+\lambda-\frac{n^2}{r^2} \right)f_n(t)=0,$$
and the same holds for $g_n$. Since $f_n$ is $2\pi$-periodic, then 
$$ \frac{T^2(1+r^2)}{4\pi^2}\left(\varpi+\lambda-\frac{n^2}{r^2} \right)=m^2,$$
for some $m\in\n$ and $f_n(t)=f_{m,n}(t)=a_m\cos(m t)+b_m\sin(m t)$, $a_m,b_m\in\r$. 
Similar conclusions are obtained for   $g_n$. The eigenvalues $\lambda$ are 
 $$\lambda_{m,n}=\frac{4\pi^2 m^2}{T^2(1+r^2)}+\frac{(n^2-1)r^2+n^2 r^2}{r^2(1+r^2)}.$$ 
 Since $\lambda_{m,n}>0$ if $n\geq 1$, then   $0$ is an eigenvalue of $L$ if    $\lambda_{m,0}=0$ for some $m\in\n$. This occurs if and only if  $T=2\pi m$. Notice that $\ee_0$ is spanned by $\{\cos(mt),\sin(mt)\}$.

Let fix $T_0=2\pi$. Then $0$ is an eigenvalue if and only if $m=1$ and the  eigenspace $\ee_0$   is spanned by the eigenfunctions   $\cos (t)$ and $\sin(t)$. In order  to have $\ee_0$ of dimension $1$, we reduce the Banach space $X$ to the space of even functions of $ C^{2,\alpha}( \r/2\pi\z \times\s^1)$. This allows to discard the function $\sin(t)$, concluding that $\ee_0$ is spanned by $u_0(t,\theta)=\cos(t)$.  This proves that the condition (1) in Thm. \ref{t-bifu} is fulfilled.

Since  $  D_u \mathcal{F}(0,T_0)$ is a self adjoint elliptic operator,  standard arguments prove that  the rank of $  D_u \mathcal{F}(0,T_0)$ is $L^2$-orthogonal to $\ee_0$. Since this space is one-dimensional, the required  codimension is $1$.  This checks the condition (2) in Thm. \ref{t-bifu}.

Finally, for the third condition of Thm. \ref{t-bifu}, we need to check that   $D_T D_u \mathcal{F}(0,T_0)(u_0)\not\in \mbox{rank } D_u \mathcal{F}(0,T_0)$. We know that 
$$
D_T D_u \mathcal{F}(0,T_0)=-\frac{8\pi^2}{T_0^3}\frac{1}{1+r^2}\partial_t^2=-\frac{1}{\pi(1+r^2)}\partial_t^2.$$
Then
$$D_T D_u \mathcal{F}(0,T_0)(u_0) =  \frac{1}{\pi(1+r^2)}u_0=\frac{1}{\pi(1+r^2)}\cos t. $$
If  $D_T D_u \mathcal{F}(0,T_0)(u_0)\in \mbox{rank } D_u \mathcal{F}(0,T_0)$, then $D_T D_u \mathcal{F}(0,T_0)(u_0) $ is orthogonal to $u_0$, that is, 
$$\frac{1}{\pi(1+r^2)}\int_0^{2\pi}\int_{\s^1}  \cos^2t\, dt d\theta=0, $$
which it is not possible.

As conclusion of Thm. \ref{t-bifu}, we obtain a uniparametric family of   surfaces $C_r((\widetilde{w}_s)_T)$ bifurcating from the cylinder $C_R$, all of them with constant mean curvature $H_0$ and different from $C_R$.  Each of these surfaces $C_r((\widetilde{w}_s)_T)$  obtained by perturbing $C_R$ are parametrized by $t$ and independent of the parameter $\theta$. This proves that they are surfaces of revolution about the $z$-axis, the same rotation axis than $C_R$. We summarize these arguments in the following result. 
 
 \begin{theorem}\label{t-b} Let $C_R$ be a Killing cylinder of radius $R$. Then the cylinder $C_R$ bifurcates in a family of surfaces of revolution about the $z$-axis, all them having the same (constant) mean curvature   than $C_R$ and being $2\pi$-periodic in the $t$-parameter.  
 \end{theorem}

By uniqueness, we know that these surfaces belong necessarily to the Delaunay surfaces of $\h^3$. We point out that  $\lambda_{m,0}=0$ for some $m$  if and only if  $T=2m\pi$, $m\in\n$. Therefore in previous arguments, we   could have chosen a different value of $T$, say, $T=2m\pi$.   In such a case, the same argument proves  the existence of cmc rotational surfaces whose period is $T$. This   obviously gives Delaunay surfaces again, in fact the same surfaces of Thm. \ref{t-b} scaled by the factor of $m$.
 
\section*{Acknowledgements} Antonio Bueno has been partially supported by CARM, Programa Regional de Fomento de la Investigaci\'on, Fundaci\'on S\'eneca-Agencia de Ciencia y Tecnolog\'{\i}a Regi\'on de Murcia, reference 21937/PI/22.

 Rafael L\'opez  is a member of the IMAG and of the Research Group ``Problemas variacionales en geometr\'{\i}a'',  Junta de Andaluc\'{\i}a (FQM 325). This research has been partially supported by MINECO/MICINN/FEDER grant no. PID2020-117868GB-I00,  and by the ``Mar\'{\i}a de Maeztu'' Excellence Unit IMAG, reference CEX2020-001105- M, funded by 

MCINN/AEI/10.13039/501100011033/ CEX2020-001105-M.


\end{document}